\documentclass[a4paper,11pt]{article}
\usepackage{amsmath}
\usepackage{amsfonts}
\usepackage{amssymb}
\usepackage{amsthm}

\theoremstyle{plain}
\newtheorem{theorem}{Theorem}
\newtheorem{proposition}[theorem]{Proposition}
\newtheorem{question}[theorem]{Question}

\newcommand{\prop}[1]{Proposition~\ref{#1}}
\newcommand{\qn}[1]{Question~\ref{#1}}

\newcommand{\defn}{\emph}
\newcommand{\N}{\mathbb{N}}
\newcommand{\Z}{\mathbb{Z}}
\newcommand{\Q}{\mathbb{Q}}

\newcommand{\first}{1}
\newcommand{\second}{2}

\usepackage{calc}

\title{Partition regularity with congruence conditions}
\author{Ben Barber
        \footnote{Department of Pure Mathematics and Mathematical Statistics, Centre for Mathematical
Sciences, Wilberforce Road, Cambridge, CB3 0WB, UK.  {\tt b.a.barber@dpmms.cam.ac.uk}}
        \and
        Imre Leader
        \footnote{Department of Pure Mathematics and Mathematical Statistics, Centre for Mathematical
Sciences, Wilberforce Road, Cambridge, CB3 0WB, UK.  {\tt i.leader@dpmms.cam.ac.uk}}}

\begin{document}

\maketitle

\begin{abstract}
An infinite integer matrix $A$ is called image partition regular if, whenever the
natural numbers are finitely coloured, there is an integer vector $x$ such
that $Ax$ is monochromatic. Given an image partition regular matrix $A$, can we
also insist that each variable $x_i$ is a multiple of some given $d_i$?
This is a question of Hindman, Leader and Strauss.

Our aim in this short note is to show that the answer is negative. As an
application, we disprove a conjectured equivalence between the
two main forms of partition regularity, namely image partition
regularity and kernel partition regularity.
\end{abstract}

\section{Introduction}

We call an integer matrix $A$ (possibly infinite but with only finitely many non-zero entries in each row) \defn{image partition regular} if, whenever $\N$ is finitely coloured, there is a vector $x$ with integer entries such that every entry of $Ax$ is in $\N$ and has the same colour (``$Ax$ is monochromatic'').  For example, Schur's theorem states that whenever we finitely colour the natural numbers we can find $x$ and $y$ such that $x$, $y$ and $x+y$ are all the same colour: that is, the matrix
\[
 \begin{pmatrix}
  1 & 0 \\
  0 & 1 \\
  1 & 1
 \end{pmatrix}
\]
is image partition regular.\footnote{What we call `image partition regular' is sometimes called `image partition regular over $\Z$' (or `weakly image partition regular') to indicate that the variables take values in $\Z$.}

In the finite case, partition regularity is very well understood (see \cite{Rado}, \cite{hl} and, for a general overview, \cite{Graham-Rothschild-Spencer}).  In contrast, in the
infinite case the situation is far less clear, and several very basic 
questions remain unanswered---see the survey \cite{problem-survey}.  Our aim in this note is
to give negative answers to three questions from \cite{problem-survey}.

\begin{question}[{\cite[Q8]{problem-survey}}]\label{q1}
Let $A$ be an image partition regular matrix, and $(d_i)$ be a sequence of integers.  Is it true that, whenever the natural numbers are finitely coloured, there is a monochromatic image $Ax$ such that the variables $x_i$ satisfy $x_i \equiv 0 \pmod{d_i}$?
\end{question}

[We remark that the question in \cite{problem-survey} does not take quite this form.  It is phrased in terms of `kernel partition regularity', and is restricted to systems for which the variables can always be taken to have the same colour as $Ax$.  Our counterexample for \qn{q1} will also be a counterexample for this more specialised question.]

Note that of course we cannot ask for $x_i \equiv a_i \pmod{d_i}$ for arbitrary $a_i$.  For example, in Schur's theorem we cannot ask for $x$ to be odd; for then colouring the odd numbers red and the even numbers blue forces all of $x$, $y$ and $x+y$ to be odd.

A related notion is that of kernel partition regularity.  We call a matrix $A$ \defn{kernel partition regular} if, whenever the natural numbers are finitely coloured, there is a monochromatic vector $x$ such that $Ax = 0$.  It is natural to ask whether the two
notions of image and kernel partition regularity are `essentially the same',
in a sense we now describe.

Given a matrix $A$ (with only finitely many non-zero entries in each row) we define a matrix $B(A)$, expressing the linear dependences between the rows of $A$, as follows.  Let $\{r_i : i \in I\}$ be a maximal linearly independent set of rows of $A$, and write each of the remaining rows $\{s_j : j \in J\}$ as a linear combination of the $r_i$.  Let $B(A)$ be the matrix corresponding to these linear equations.  That is, for each $j \in J$ write $s_j = \sum_{i \in I} c_{ji}r_i$, and let $B(A)$ be the $J \times (I \cup J)$ matrix with left-hand side $(c_{ji})$ and right-hand side $-1$ times the $J \times J$ identity matrix.\footnote{The reader may be worried about the case when the rows of $A$ are linearly independent.  In that case $B(A)$ is the empty matrix, which is automatically kernel partition regular.}

If $A$ is image partition regular then certainly $B(A)$ is kernel partition regular, because any monochromatic image of $A$ is in the kernel of $B(A)$.  Conversely:

\begin{question}[{\cite[Q7]{problem-survey}}]\label{q2}
If $B(A)$ is kernel partition regular, must $A$ be image partition regular?
\end{question}

Finally, we say that an integer matrix $A$ (with only finitely many non-zero entries in each row) is \defn{image partition regular over $\Q$} if, whenever $\Q\setminus\{0\}$ is finitely coloured, there is a vector $x$ with entries in $\Q$ such that $Ax$ is monochromatic.

\begin{question}[{\cite[Q9]{problem-survey}}]\label{q3}
If $A$ is image partition regular over $\Q$, must it be image partition regular?
\end{question}

The corresponding question for kernel partition regularity was answered in \cite{bhl}.

Our aim in this short note is to show that the answer to each of these questions is negative.  The three questions turn out to be related: once we have a counterexample for \qn{q1}, counterexamples for \qn{q2} and \qn{q3} will follow.

\section{The counterexample}

We seek a matrix $A$, and a sequence of natural numbers $(d_i)$, such that $A$ is image partition regular but we cannot always find a monochromatic image $Ax$ satisfying $x_i \equiv 0 \pmod{d_i}$ for each $i$.  The matrix we shall make use of comes from a recent result of Barber, Hindman and Leader \cite{bhl}.  It is simpler to give the system of expressions corresponding to its rows than to describe the image partition regular matrix explicitly.  In the following system, the variables are $y$ and $x_{ij}$, $1 \leq j \leq i$.

\begin{theorem}[{\cite[Theorem 15]{bhl}}] \label{ipr-system}
 Let $(c_n)$ be any sequence of integer coefficients.  Then the system of expressions
 \[ \renewcommand*{\arraystretch}{1.4}
  \setlength{\arraycolsep}{2em}
  \begin{array}{rcc}
                    x_{11} + c_1  y  & x_{11} & y \\
           x_{21} + x_{22} + c_2  y  & x_{21} & \\
  x_{31} + x_{32} + x_{33} + c_3  y  & x_{22} & \\
                               {\mathrel{\makebox[\widthof{$c_3 y + c_3 y$}]{\vdots}}}  & \vdots  &
  \end{array} \tag{\first}
  \label{counterexample}
 \]
 is image partition regular.
\end{theorem}

[The result in \cite{bhl} was stated for $c_n=2^n$, but, as remarked in \cite{bhl}, the proof works for any integer sequence $(c_n)$.]

Roughly speaking, our idea is to `simulate' the condition that a variable be
odd by giving some conditions on the other variables. This will be achieved
by a particular choice of the constants $c_n$ and divisibility constraints $d_{ij}$.

For every $n \in \N$ we choose $c_n \in \N$ such that
\[
 c_n n \equiv 2^{n-1} \pmod{2^n}.
\]
To see that this is possible, write $n=2^kp$, where $p$ is odd and $k < n$.  Then we seek $c_n$ such that
\[
 c_n p \equiv 2^{n-k-1} \pmod {2^{n-k}}.
\]
But $p$ is odd, hence invertible mod~$2^{n-k}$, so these congruences have solutions.  For example, we can take $c_1 = 1$, $c_2 = 1$, $c_3 = 4$, $c_4 = 2$ and $c_5 = 16$.

We now show that the system \eqref{counterexample} with this choice of coefficients provides a counterexample for \qn{q1}.

\begin{proposition} \label{N-colouring}
 There is a 2-colouring of $\N$ such that there are no integers $y$ and $x_{ij}$ such that $x_{ij} \equiv 0 \pmod {2^i}$ and the system \eqref{counterexample} is monochromatic.
\end{proposition}
\begin{proof}
 Note first that all the variables must be positive, as they belong to the image, which is in $\N$.  We shall define the colouring in stages, so that at the $n$th stage we force $n$ and the $n$th expression in the first column to have different colours.  This rules out the possible values of $y$ one by one.

 For $n=1$ we colour 1 red and all the other odd numbers blue.  Since $x_{11}$ must be even and positive, and $c_1$ is odd, the first expression must be blue.

 For $n=2$ we colour 2 red and all the other numbers that are 2 mod 4 blue.  Since $x_{21}$ and $x_{22}$ are 0 mod 4 and positive, and $2 c_2 \equiv 2 \pmod 4$, the second expression must be blue.

 For $n=3$, the number 3 has already been coloured blue.  So we colour every number that is 4 mod 8 red.  Since $x_{31} \equiv x_{32} \equiv x_{33} \equiv 0 \pmod 8$, and $3 c_3 \equiv 4 \pmod 8$, the third expression must be red.

 Continue.  At the $n$th stage, $n$ has already received some colour.  Give all numbers which are $2^{n-1}$ mod $2^n$ the opposite colour to $n$.  Since $x_{n1} \equiv \cdots \equiv x_{nn} \equiv 0 \pmod {2^n}$, and $c_n n \equiv 2^{n-1} \pmod{2^n}$, the $n$th expression must have the opposite colour to $n$.
\end{proof}

We now turn to \qn{q2}.  We will use a reformulation of \prop{N-colouring} to obtain a counterexample.  Consider the system below, obtained by reparameterising \eqref{counterexample} by setting $x_{ij} = 2^iz_{ij}$.
 \[ \renewcommand*{\arraystretch}{1.4}
  \setlength{\arraycolsep}{2em}
  \begin{array}{rcc}
                                 2z_{11} + c_1 y & 2z_{11} & y \\
                      4z_{21} + 4z_{22} + c_2 y & 4z_{21} \\
           8z_{31} + 8z_{32} + 8z_{33} + c_3 y & 4z_{22} \\
                                           {\mathrel{\makebox[\widthof{$c_n y + c_n y$}]{\vdots}}} & \vdots \\
         2^nz_{n1} + \ldots + 2^nz_{nn} + c_n y & 2^iz_{ij} \\
                                           {\mathrel{\makebox[\widthof{$c_n y + c_n y$}]{\vdots}}} & \vdots
  \end{array} \tag{\second}
  \label{ZvQ}
 \]
The system \eqref{ZvQ} is not image partition regular, as if it were then taking $x_{ij} = 2^iz_{ij}$ would contradict \prop{N-colouring}.  But if \eqref{counterexample} is represented by the matrix $A_\first$ and \eqref{ZvQ} is represented by the matrix $A_\second$, then $B(A_\first) = B(A_\second)$, so image partition regularity of $A$ cannot be determined by examining the matrix $B(A)$.  (In fact, $B(A_\first)$ is kernel partition regular, which is the main result of \cite{bhl}.)

Finally, we observe that \eqref{counterexample} and \eqref{ZvQ} have the same images over $\Q$.  So \eqref{ZvQ} is not image partition regular but is image partition regular over $\Q$, providing a counterexample for \qn{q3}.


\begin{thebibliography}{9}

\bibitem{bhl}
Ben Barber, Neil Hindman, and Imre Leader.
\newblock Partition regularity in the rationals.
\newblock {\em J. Combin. Theory Ser. A}, 120:1590--1599, 2013.

\bibitem{Graham-Rothschild-Spencer}
R.L. Graham, B.L. Rothschild, and J.H. Spencer.
\newblock {\em Ramsey theory}.
\newblock Wiley-Interscience Series in Discrete Mathematics and Optimization.
  John Wiley \& Sons Inc., New York, second edition, 1990.

\bibitem{problem-survey}
N.~Hindman, I.~Leader, and D.~Strauss.
\newblock Open problems in partition regularity.
\newblock {\em Combin. Probab. Comput.}, 12:571--583, 2003.

\bibitem{hl}
Neil Hindman and Imre Leader.
\newblock Image partition regularity of matrices.
\newblock {\em Combin. Probab. Comput.}, 2(4):437--463, 1993.

\bibitem{Rado}
R.~Rado.
\newblock Studien zur {K}ombinatorik.
\newblock {\em Math. Z.}, 36:424--470, 1933.

\end{thebibliography}
\end{document}